\documentclass[11pt]{article} 
\usepackage{algorithm, algorithmic,wrapfig}

\usepackage[left=1in,top=1in,right=1in,bottom=1in,letterpaper]{geometry}
\usepackage{microtype}
\usepackage{graphicx}
\usepackage{subfigure}
\usepackage{booktabs,amsmath} % for professional tables
\usepackage{multirow,xcolor}
\usepackage{soul}
\usepackage{caption}
\providecommand{\keywords}[1]{\textbf{\textit{Keywords:}} #1}

\usepackage{hyperref}
 
\newcommand{\blue}[1]{\textcolor{blue}{#1}}

\allowdisplaybreaks

\usepackage{comment,url,algorithm,graphicx,relsize}
\usepackage{amssymb,amsfonts,amsmath,amsthm,amscd,dsfont,mathrsfs,mathtools,nicefrac}
\usepackage{float,psfrag,epsfig,color,xcolor,url}
\usepackage{epstopdf,bbm,mathtools,enumitem}

\def\reals{\mathbb{R}} % Real number symbol

\def\norm#1{\left\|{#1}\right\|} % A norm with 1 argument
\newcommand{\inner}[1]{{\langle #1 \rangle}} % inner product

\def\normal{{\sf N}}

\newcommand{\grad}{\nabla}

\newcommand{\E}{\mathbb{E}}

\providecommand{\argmin}{\mathop\mathrm{arg min}}

\ifdefined\nonewproofenvironments\else
\ifdefined\ispres\else
\newtheorem{theorem}{Theorem}
\newtheorem{lemma}[theorem]{Lemma}

\renewenvironment{proof}{\noindent\textbf{Proof.}\hspace*{.3em}}{\qed\\}
\newenvironment{proof-sketch}{\noindent\textbf{Proof Sketch}
  \hspace*{0.em}}{\qed\bigskip\\}
\newenvironment{proof-idea}{\noindent\textbf{Proof Idea}
  \hspace*{0.em}}{\qed\bigskip\\}
\newenvironment{proof-of-lemma}[1][{}]{\noindent\textbf{Proof of Lemma {#1}.}
  \hspace*{0.em}}{\qed\\}
\newenvironment{proof-of-corollary}[1][{}]{\noindent\textbf{Proof of Corollary {#1}.}
  \hspace*{0.em}}{\qed\\}
\newenvironment{proof-of-theorem}[1][{}]{\noindent\textbf{Proof of Theorem {#1}.}
  \hspace*{0.em}}{\qed\\}
\newenvironment{proof-attempt}{\noindent\textbf{Proof Attempt}
  \hspace*{0.em}}{\qed\bigskip\\}

\fi
\newtheorem{proposition}[theorem]{Proposition}

\newtheorem{assumption}{Assumption}
\newcommand\numberthis{\addtocounter{equation}{1}\tag{\theequation}}

\title{Improved Complexities for Stochastic Conditional Gradient Methods under Interpolation-like Conditions}

\author{Tesi Xiao\thanks{Department of Statistics, University of California, Davis. \texttt{texiao@ucdavis.edu}.  } 
\and Krishnakumar Balasubramanian\thanks{Corresponding author. One Shields Avenue, Department of Statistics, University of California, Davis, CA 95616 \texttt{kbala@ucdavis.edu}. } 
\and Saeed Ghadimi\thanks{Department of Management Sciences, University of Waterloo. \texttt{sghadimi@uwaterloo.ca}.}
}

\begin{document}
\maketitle

\begin{abstract}
We analyze stochastic conditional gradient methods for constrained optimization problems arising in over-parametrized machine learning. We show that one could leverage the interpolation-like conditions satisfied by such models to obtain improved oracle complexities. Specifically, when the objective function is convex, we show that the conditional gradient method requires $\mathcal{O}(\epsilon^{-2})$ calls to the stochastic gradient oracle to find an $\epsilon$-optimal solution. Furthermore, by including a gradient sliding step, we show that the number of calls reduces to $\mathcal{O}(\epsilon^{-1.5})$.%We also establish similar improved results in the zeroth-order setting, where only noisy function evaluations are available. %Notably, the above results are achieved without any variance reduction techniques, thereby demonstrating the improved performance of vanilla versions of conditional gradient methods for over-parametrized machine learning problems.
\end{abstract}


\section{Introduction}
Consider the following constrained stochastic optimization problem: 
\begin{equation}\label{eq:sopt}
    \underset{x \in \Omega}{\min}\quad \left\{f(x):= \E_\xi \left[ F(x, \xi)\right] \right\},
\end{equation}
where $f: \mathbb{R}^d \to \mathbb{R}$ and $\Omega \subset \mathbb{R}^d$ is a closed and convex set and $\xi$ is a random vector characterizing the stochasticity in the problem. In a machine learning setup, the function $F$ could be interpreted as the loss function associated with a sample $\xi$ and the function $f$ could represent the risk, which is defined as the expected loss. Such constrained stochastic optimization problems arise frequently in statistical machine learning applications. The conditional gradient algorithm,  also called as the Frank-Wolfe algorithm, is an efficient method for solving constrained optimization problems of the form in~\eqref{eq:sopt} due to their projection-free nature~\cite{jaggi2013revisiting, harchaoui2015conditional, freund2017extended, lan2017conditional, berrada2018deep, ravi2018constrained}. In each step of the conditional gradient method,  it is only required to minimize a linear objective over the set $\Omega$. This operation could be implemented efficiently for a variety of sets arising in statistical machine learning,  compared to the operation of projecting onto the set $\Omega$, which is required for example by the projected gradient method. Hence, the conditional gradient method has regained popularity in the last decade in the optimization and machine learning community.

There has been extensive work in the past decade on analyzing the stochastic conditional gradient algorithm for optimization problems of the form in~\eqref{eq:sopt}; see for example~\cite{garber2013linearly, hazan2016variance, lan2016conditional,reddi2016stochastic, ghadimi2019conditional}.  However, existing works do not take into account certain favorable structures that are naturally available in modern over-parametrized machine learning problems. \textcolor{black}{Specifically, it has been noted that modern machine learning models predict well on unseen data, despite fitting the training data perfectly~\cite{zhang2016understanding, huang2017densely, liang2018just, ma2018power, montanari2019generalization, hastie2019surprises}}. Examples include logistic regression or support vector machine with squared-hinge loss that are trained with linearly separable data~\cite{vaswani2019fast,vaswani2019painless,meng2019fast} and deep neural networks~\cite{vaswani2019fast,bassily2018exponential}. From an optimization point of view, for the problem in~\eqref{eq:sopt} with $\Omega\equiv\mathbb{R}^d$, the above interpolation condition means that at the optimal point, the gradient is not only zero (or close to zero) with respect to the risk function $f$ but is also almost surely equal to zero for the random loss function $F$. Such a scenario helps to reduce the stochasticity in the gradient estimation process which in turn results in improved complexity results for several stochastic optimization procedures. Indeed in the recent past, several works have provided improved rates for algorithms like stochastic gradient descent~\cite{needell2014stochastic, ma2018power, bassily2018exponential, gower2019sgd, vaswani2019fast, vaswani2019painless} and sub-sampled Newton's method~\cite{meng2019fast}. In particular, for several settings, the above works demonstrate that the stochastic algorithm may perform as well as the corresponding deterministic counterpart. However, such works only study unconstrained optimization problems and do not have any consequences for constrained stochastic optimization problems of the form in~\eqref{eq:sopt}.

Hence, in this work we consider the following question: \emph{Can we obtain improvements in the oracle complexity of algorithms used for projection-free constrained stochastic optimization problems arising in the context of \textcolor{black}{over-parametrized machine learning models, that are capable of perfectly interpolating the training data}?}  We give a positive answer to the above question by demonstrating that the stochastic conditional gradient method, a projection-free technique for solving constrained stochastic optimization problems, also enjoy improved oracle complexities when they are used to solve constrained stochastic optimization problems of the form in~\eqref{eq:sopt} under certain \emph{interpolation-like} conditions. We elaborate on the specific form of improvement observed below. For stochastic conditional gradient algorithms, the oracle complexity is measured in terms of number of calls to the Stochastic First-order Oracle (SFO) and the Linear Minimization Oracle (LMO) used to the solve the subproblems \textcolor{black}{(that are of the form of minimizing a linear function over the convex feasible set)} arising in the algorithm. \textcolor{black}{In this work, we make the following contribution to the literature on conditional gradient methods under interpolation-like assumptions (see Section~\ref{sec:assumptions} for the exact definitions) on the stochastic gradient:}
\begin{enumerate}[leftmargin=20pt,noitemsep]
\item For the case of convex $f$ in~\eqref{eq:sopt}, we show that the number of calls to the SFO for the \emph{vanilla} stochastic conditional gradient method and stochastic conditional gradient sliding methods are given respectively by $\mathcal{O}(\epsilon^{-2})$ and $\mathcal{O}(\epsilon^{-1.5})$. For comparison, without such assumptions, the corresponding complexities are $\mathcal{O}(\epsilon^{-3})$ and $\mathcal{O}(\epsilon^{-2})$ respectively. The number of calls to the linear minimization oracle (LMO) is of the order $\mathcal{O}(\epsilon^{-1})$, in both cases.
\item We also demonstrate similar improvements in the context of zeroth-order conditional gradient methods, where one only observes noisy evaluations of the function being optimized. Specifically, the number of calls to the stochastic zeroth-order oracle for the \emph{vanilla} stochastic conditional gradient method and stochastic conditional gradient sliding methods are given respectively by $\mathcal{O}(d\epsilon^{-2})$ and $\mathcal{O}(d\epsilon^{-1.5})$, with the same LMO complexity as the first-order setting.
\end{enumerate}
We emphasize that, notably the above improvements are achieved without incorporating any double-loop based existing variance reduction techniques, for example SVRF~\cite{reddi2016stochastic} or SPIDER-FW~\cite{yurtsever2019conditional}. It is also worth noting that ~\cite{defazio2019ineffectiveness,schmidt2020} argue that variance reduction techniques (at the least existing ones) are ineffective in the context of modern deep learning models which are invariably over-parametrized. We also remark that, in contrast to stochastic gradient methods for unconstrained optimization~\cite{vaswani2019fast, bassily2018exponential}, the above improved results  still do not match the corresponding deterministic rates highlighting the subtlety with projection-free optimization.

\section{Preliminaries and Assumptions}\label{sec:assumptions}
We now list and discuss the set of assumptions made in our work. We first list some regularity assumptions on the function $f$ and the set $\Omega$.
\begin{assumption}\label{ass:constrained}
The function $f$ has L-Lipschitz gradient $\grad f$, i.e., for any pair of points $x,y\in \Omega$, we have $ \norm{\grad f(x)- \grad f(y)} \leq L \norm{x-y}$, and the feasible set $\Omega\subset\reals^d$ is bounded, i.e.,$\underset{x,y \in \Omega}{\max}  ~\norm{x-y}\leq D$.
\end{assumption}
\noindent The above set of assumptions are standard in the analysis of stochastic conditional gradient methods and has been used in prior works in the literature; see for example~\cite{ghadimi2013stochastic}. We make the above assumptions for both the first-order setting. We also require the following smoothness assumption in the zeroth-order setting.
\begin{assumption}\label{ass:smooth-F}
The function $F$ has Lipschitz continuous gradient with constant $L$, almost surely for any $\xi$, i.e., for any $x,y\in \mathbb{R}^d$, i.e., almost surely we have $\norm{\grad F(x,\xi) - \grad F(y,\xi)} \leq L\norm{x-y}$.
\end{assumption}

Note that the above assumption is stronger than the first statement of Assumption~\ref{ass:constrained} and implies it. However, we only use Assumption~\ref{ass:smooth-F} for the analysis of zeroth-order algorithms.

\subsection{Growth Conditions in the Convex Constrained Setting}
We now state the main interpolation-like assumptions that we make in our work when $f$ is convex and provide the main intuition behind such an assumption.

\begin{assumption}[Moment-based Weak Growth Condition] \label{ass:WGC-moment}
Let $x^*$ be the minimum point of $f$. We say that $f$ satisfies the Moment-based Weak Growth Condition (WGC) with constant $\rho$, if for any point $x\in \Omega$, we have
\begin{equation}
    \E_\xi \norm{\grad F(x,\xi)}^2 \leq 2\rho L \left[f(x)-f(x^*)\right].
\end{equation}
\end{assumption}

\begin{assumption}[Variance-based Weak Growth Condition]\label{ass:WGC-variance}
Let $x^*$ be the minimum point of $f$. We say that the function $f$ satisfies the Variance-based Weak Growth Condition (WGC) with constant $\rho$, if for any point $x\in \Omega$, we have
\begin{equation}
    \E_\xi \norm{\grad F(x,\xi)  - \grad f(x) }^2 \leq 2\rho L \left[f(x)-f(x^*)\right].
\end{equation}
\end{assumption}

The above conditions are motivated by the so-called strong growth condition: $\E \|\nabla F(x, \xi) \|^2 \leq \rho \| \nabla f(x) \|^2$, used in \cite{vaswani2019fast} for obtaining faster rates of convergence for stochastic gradient method in the unconstrained setting. Notice that in the interpolation setting, when $\nabla f(x^*) = 0$, we have $\nabla F(x^*, \xi) = 0$, almost surely. Thus, the strong growth condition is defined exactly to take advantage of this situation. Furthermore, in the smooth convex setting,~\cite{vaswani2019fast} showed that the strong-growth condition is equivalent to the moment-based weak growth condition in Assumption~\ref{ass:WGC-moment}. However, the moment-based weak growth condition as proposed in \cite{vaswani2019fast} is not directly suited for the constrained stochastic setting that we consider in this work. It is easy to construct examples for which there exists stationary point at the boundary of $\Omega$ with non-zero (stochastic) gradient, i.e., $\E \|\nabla F(x, \xi) \|^2 $ could remain positive while the right hand side goes to $0$ and hence the assumption is not satisfied. In order to resolve this issue, for the constrained setting, we relax the moment-based growth conditions to the variance-based versions. Note that we have 
\begin{equation*}
\E \| \nabla F(x, \xi) - \nabla f(x)  \|^2 = \E \| \nabla F(x, \xi) \|^2 - \|  \nabla f(x) \|^2 \leq \E \| \nabla F(x, \xi) \|^2.
\end{equation*}
Thus variance-based growth conditions naturally become the substitute for the moment-based version in constrained problems and could hold even the moment-based conditions do not hold. As they are also motivated by the interpolation assumption, we refer to these conditions as interpolation-like conditions. Formally, under the variance-based growth conditions for a convex $f$, if we attain \blue{an} optimal point $x^*\in\Omega$, the variance of the stochastic first-order oracle will be almost surely zero, i.e., $\nabla F(x^*, \xi) = \nabla f(x^*)$ almost surely. This property eventually leads to the improvements in the query complexity that we demonstrate. We emphasize that it is natural to construct counter-examples that violate Assumption~\ref{ass:WGC-variance}. In those cases, the improved query complexities that we demonstrate are simply not applicable. Finally, we also have the following natural relationships between the two conditions.

\begin{proposition}\label{prop:WGC}
The Weak Growth Conditions defined above have the following relations: 
\begin{itemize}[leftmargin=0.2cm, align=left, topsep=1pt, partopsep=1pt, itemsep=1pt,parsep=1pt]
	\item[(a)] If $f$ satisfies the Moment-based WGC \eqref{ass:WGC-moment} with $\rho$, then $f$ satisfies the Variance-based WGC \eqref{ass:WGC-variance} with $\rho$ and there exists $x^*\in \Omega$ such that $\nabla f(x^*) = 0$.
	\item[(b)] If $f$ satisfies the Variance-based WGC \eqref{ass:WGC-variance} with $\rho$ and there exists $x^*\in \Omega$ such that $\nabla f(x^*) = 0$, then $f$ satisfies the Moment-based WGC \eqref{ass:WGC-moment} with $\rho+1$.
\end{itemize}
\end{proposition}

\subsection{Growth Conditions in the Zeroth-Order Constrained Setting}

In the zeroth-order setting, we only assume availability of the noisy function evaluations. This oracle setting is motivated by several applications where only noisy function queries of problem \eqref{eq:sopt} is available, such as reinforcement learning~\cite{salimans2017evolution, choromanski2018structured, choromanski18a}, hyperparameter tuning~\cite{snoek2012practical}, and black-box attacks to deep networks \cite{chen2017zoo, sahu2019towards}. Hence, we use the Gaussian Stein's identity based random gradient estimator, a standard gradient estimator in the zeroth-order optimization literature~\cite{ghadimi2013stochastic, duchi2015optimal, nesterov2017random,balasubramanian2018zeroth}:
\begin{equation*}
\displaystyle
    \bar{G}_{\nu}(x) = \frac{1}{b}  \sum_{j=1}^{b} \frac{F(x + \nu u_{j}, \xi_{j}) -  F(x, \xi_{j}) }{\nu} u_{j},
\end{equation*}
where $u_{1},\dots, u_{b}$ are i.i.d. samples from $\normal(0, \mathbf{I}_d)$. The above gradient estimator is a biased estimator of the true gradient $\nabla f(x)$, and was also used in \cite{balasubramanian2018zeroth}, to develop zeroth-order conditional gradient descent algorithms.

While for the first-order setting, we use the relatively weaker variance-based conditions to obtain the improved bounds, in the zeroth-order setting, it turns out the stronger moment-based conditions are required. The reason is that the mean square error of the biased zeroth-order gradient estimator is bounded above by $\E \| \grad F(x,\xi) \|^2$. Hence, to obtain improved rates, it makes it necessary to make assumptions on the moments of the stochastic gradient directly. We emphasize that this is required only for the constrained problems, since the moment-based conditions are equivalent to the variance-based conditions when there exists one zero-gradient point in the constraint set (see Proposition~\ref{prop:WGC}). In particular, we show in Appendix~\ref{sec:sgdproofs} that a zeroth-order version of Theorem 3 from~\cite{vaswani2019fast}, for stochastic gradient descent, to bound the gradient size in the nonconvex setting could be proved just under the variance-based growth conditions. 

\iffalse
To summarize, for the zeroth-order convex setting, we use Assumption~\ref{ass:WGC-moment}. For the zeroth-order nonconvex setting, we use the following assumption, which is the moment version of the variance-based condition in Assumption~\ref{ass:CGC-variance}.

\begin{assumption}\label{ass:CGC-moment}
(Moment-based Constrained Growth Condition) Let $x^*$ be a stationary point. The function $f$ satisfies the Moment-based Constrained Growth Condition (CGC) with constant $\rho$, if for any point $x\in \Omega$,
\begin{equation}
    \E_\xi \norm{\grad F(x,\xi)}^2 \leq 2\rho L \mathcal{G}_f (x).
\end{equation}
\end{assumption}
Note that moment-based condition also posits that there is no stationary points at boundary. 
\fi

\subsection{Motivating Examples}
Before we present our main results in the next section, we briefly discuss some motivating examples of constrained stochastic optimization problems that arise in modern machine learning. In the convex setting, it is easy to see that kernel regression~\cite{liang2018just}, squared-Hinge loss based linear SVM classifier or logistic regression on linearly separable data could be considered as operating in the over-parametrized regime and hence satisfy interpolation-like conditions~\cite{vaswani2019fast, meng2019fast}. 
 
However, without any constraints, such predictors might be biased against certain sensitive features like race or gender. One way to build fair predictors is to explicitly encode fairness constraints with respect to certain pre-defined sensitive features~\cite{donini2018empirical,agarwal2018reductions}. Specifically, it was shown in~\cite{agarwal2018reductions} that several standard and well-accepted notions of fairness  in classification setting, including equalized odds~\cite{hardt2016equality}, demographic parity~\cite{dwork2018decoupled}, balance for the negative class~\cite{kleinberg2016inherent}, treatment equality~\cite{berk2018fairness} could be formulated as empirical risk minimization problems subjected linear inequality constraints. In this case, the problem is exactly of the form in~\eqref{eq:sopt} with $\Omega$ being a polytope. Furthermore,~\cite{donini2018empirical} also proposed a general approach for fair empirical risk minimization. Similar to~\cite{agarwal2018reductions}, the fundamental idea is to enforce constraints such that the conditional risk of a predictor is not varying much with respect to the sensitive features associated with the problem. Such formulations of fair empirical risk minimization in the interpolation regime also fall under the class of problems in~\eqref{eq:sopt}.

\textbf{Squared hinge loss with linearly separable data.} As a concrete example, we extend the unconstrained examples presented in \cite{vaswani2019fast} to the constrained setting we consider. Assuming a finite support of features and the linearly separable data, it has been shown that the squared-hinge loss satisfies SGC with $\rho = {c}/{\tau^2}$ where $c$ is the cardinality of the support and $\tau$ is the margin (Lemma 1 in \cite{vaswani2019fast}). In the above regime, the optimal classifier that minimizes the loss and achieves a stationary point with zero gradient is not always unique. In practice, to construct a fair classifier, enforcing constraints is a natural approach. Note that if there exists an $x^*\in \Omega$, by the convexity and the L-smoothness of $f$, we have
\begin{equation}
    \| \nabla f(x) \|^2 \leq 2L(f(x) - f(x^*)).
\end{equation}
That is to say, for linearly separable data with margin $\tau$ and a finite support of size $c$, if there exists one $x^*\in\Omega$, the squared-hinge loss satisfies Assumption \ref{ass:WGC-moment} with $\rho ={c}/{\tau^2}$.

\section{Improved Complexities for Stochastic Conditional Gradient Methods}
We now provide improved complexities for stochastic conditional gradient methods under the interpolation-like assumption in Section~\ref{sec:assumptions}. For convenience, we first introduce the following mini-batch stochastic gradients with first-order and zeroth-order oracle access: at $t$-th iteration, we uniformly pick i.i.d. samples $\{\xi_{t,1}, \dots, \xi_{t, b_t}\}$ and  estimate the gradient by
\begin{equation*}
	\Tilde{\nabla}_t := \frac{1}{b_t} \sum_{i=1}^{b_t} \nabla F(x_{t-1}, \xi_{t,i}), \quad
	\bar{G}_{\nu}^{t} := \frac{1}{b_t}  \sum_{j=1}^{b_t} \frac{F(x_{t-1} + \nu u_{t,j}, \xi_{t,j}) -  F(x_{t-1}, \xi_{t,j}) }{\nu} u_{t,j}
\end{equation*}
where $u_{t,1},\dots, u_{t.b_t}$ are i.i.d. samples from $\normal(0, \mathbf{I}_d)$.

\subsection{Stochastic Frank-Wolfe}

In this section, we studied the oracle complexity of the vanilla stochastic Frank-Wolfe algorithm under the weak interpolation-like conditions in Assumption~\ref{ass:WGC-variance} and~\ref{ass:WGC-moment}.

\begin{algorithm}[tb]
\caption{Stochastic Frank-Wolfe}
\begin{algorithmic}
	\STATE \textbf{Input:} $x_0\in \Omega$, number of iterations $T$, $\gamma_t \in [0,1]$, minibatch size $b_t$
    \FOR{$t = 1, 2, \dots, T$}
        %\STATE Uniformly randomly pick i.i.d. samples $\{\xi_{t,1}, \cdots, \xi_{t,b_t}\}$
        \STATE Compute the gradient $g_t$ as follows:
        \begin{itemize}[topsep=1pt, partopsep=1pt, itemsep=1pt,parsep=1pt]
        \item[]  Set $g_t = \Tilde{\nabla}_t$ (for the first-order setting).
        \item[] Set $g_t = \bar{G}_{\nu}^{t}$ (for the zeroth-order setting).
        \end{itemize}
        \STATE Compute $d_t = \argmin_{d\in \Omega}  \inner{d, g_t}$
        \STATE $x_{t} = x_{t-1} + \gamma_t(d_t-x_{t-1})$
    \ENDFOR
    \STATE \textbf{Output:} $x_T$
\end{algorithmic}
\label{alg:SFWFandZ}
\end{algorithm}

\begin{theorem}\label{thm:sfw-convex-wgc}
Consider solving problem~\eqref{eq:sopt}, by Algorithm~\ref{alg:SFWFandZ}, under Assumption~\ref{ass:constrained} with $f$ being convex.
\begin{itemize}[leftmargin=0.2cm,align=left, topsep=1pt, partopsep=1pt, itemsep=1pt,parsep=1pt]
	\item[(a)] Assuming access to stochastic first-order oracle, under Assumption~\ref{ass:WGC-variance},
setting
\begin{equation*}
	\gamma_t = \frac{4}{t+3},\quad b_t = \lceil (t+3)/2 \rceil,
\end{equation*}
we have the following convergence rate:
\begin{equation*}
     \E[f(x_t) - f(x^*)] \leq \frac{2(f(x_0) - f(x^*))+8(\rho+1) LD^2}{t+3}.
\end{equation*}
Hence, the total number of calls to the stochastic first-order oracle and linear minimization oracle required to be
solved to find an $\epsilon$-optimal point of problem \eqref{eq:sopt} are, respectively, bounded by 
\begin{equation*}
	\mathcal{O}\left(\epsilon^{-2}\right),\quad \mathcal{O}\left(\epsilon^{-1}\right).
\end{equation*} 
\item[(b)]  Assuming access to stochastic zeroth-order oracle, under Assumptions~\ref{ass:WGC-moment}~and~\ref{ass:smooth-F}, setting 
\begin{align*}
	\gamma_t = \frac{4}{t+3},\quad  b_t = (t+3)(d+4),\quad  \nu = \frac{ D} {(T+3)(d+6)^{3/2}}
\end{align*}
we have
\begin{equation*}
	\E[f(x_t) - f(x^*)] \leq \frac{2(f(x_0) - f(x^*))+8(\rho + \rho^{-1}+1) LD^2}{t+3}.
\end{equation*}
Hence, the total number of calls to the stochastic zeroth-order oracle and linear minimization oracle required to be
solved to find an $\epsilon$-optimal point of problem \eqref{eq:sopt} are, respectively, bounded by
\begin{equation*}
	\mathcal{O}\left(d\epsilon^{-2}\right), \quad \mathcal{O}\left(\epsilon^{-1}\right).
\end{equation*}
\end{itemize}
\end{theorem}

The above oracle complexities in the first-order setting, match the results obtained by~\cite{yurtsever2019conditional, zhang2019one}. However, the above works require double-loop based variance reduction techniques which in turn require the stronger mean-square gradient-Lipschitz assumption. Furthermore, the use of the variance reduction technique results in the increased wall-clock running time of the algorithm. Our result here is applicable to the vanilla version of the stochastic conditional gradient method, as long as the problem satisfies the interpolation-like conditions observed in modern machine learning problems.

\subsection{Stochastic Conditional Gradient Sliding}
In this section, we analyze the complexity of the stochastic gradient sliding (SCGS) algorithm under the weak growth condition. The SCGS was first proposed and thoroughly analyzed in \cite{lan2016conditional}.  It is a fundamental modification of the conditional gradient algorithm that achieved improved oracle complexities without relying on any variance reduction techniques. Below, we show that under the interpolation-like assumptions in Section~\ref{sec:assumptions}, the oracle complexity of the SCGS could be further improved compared in both the first-order and zeroth-order methods.

\begin{minipage}{.45\textwidth}
\begin{algorithm}[H]
\caption{SCGS:\\Stochastic Conditional Gradient Sliding}
\begin{algorithmic}
    \STATE \textbf{Input:} $x_0\in \Omega$,  $T$,   $\beta_t\in\reals_+$, $\gamma_t\in [0,1]$,  $b_t$, $y_0=x_0$
    \FOR{$t = 1, 2, \dots, T$}
    	\STATE Set $z_t = (1-\gamma_t)x_{t-1}  + \gamma_t y_{t-1} $
    \STATE Compute the gradient $g_t$ as follows:
        \begin{itemize}[topsep=1pt, partopsep=1pt, itemsep=1pt,parsep=1pt]
        \item[]  Set $g_t = \Tilde{\nabla}_t$ (first-order).
        \item[] Set $g_t = \bar{G}_{\nu}^{t}$ (zeroth-order).
        \end{itemize}
       \STATE Solve
       \begin{itemize}[topsep=1pt, partopsep=1pt, itemsep=1pt,parsep=1pt]
        \item[]  $y_t = \text{ICG}(g_t, y_{t-1}, \beta_t, \eta_t)$ 
        \end{itemize} 
        by Algorithm~\ref{alg:ICG}
        \STATE Set $x_{t} = (1-\gamma_t)x_{t-1} + \gamma_t y_{t}$
    \ENDFOR
    \STATE \textbf{Output:} $x_T$
\end{algorithmic}
\label{alg:SFWSlideFandZ}
\end{algorithm}
\end{minipage}
\hfill
\begin{minipage}{.5\textwidth}
\begin{algorithm}[H]
\caption{ICG: \\Inexact Conditional Gradient Method}
\begin{algorithmic}
	  \STATE \textbf{Input:} $g, u, \beta, \eta,  u_1=u,  k=1$
      \STATE 1. Let $v_k$ be an optimal solution for the subproblem
      \begin{equation}\label{eq:acc_subprob}
      \max_{v\in\Omega} \{ h_k(v)=\inner{g+\beta(u_k-u), u_k-v} \}. 
      \end{equation}       
      \STATE 2. If $h_k(v_k) \leq \eta$, terminate and output $u_k$.
      \STATE 3. $u_{k+1} = (1-\alpha_k) u_k + \alpha_k v_k$ with
      $$\alpha_k = \min\bigg\{ 1, \frac{\inner{\beta(u-u_k) -g, v_t-u_t}}{\beta\norm{v_k - u_k}^2} \bigg\}.$$
      \STATE 4. Set $k\leftarrow k+1$ and go to step 1.
\end{algorithmic}
\label{alg:ICG}
\end{algorithm}
\end{minipage}

\begin{theorem}\label{thm:scgs-convex-wgc}
Consider solving problem~\eqref{eq:sopt}, by Algorithm~\ref{alg:SFWSlideFandZ}, under Assumption~\ref{ass:constrained} with $f$ being convex.
\begin{itemize}[leftmargin=0.2cm,align=left, topsep=1pt, partopsep=1pt, itemsep=1pt,parsep=1pt]
\item [(a)]  Assuming access to stochastic first-order oracle, under Assumption~ \ref{ass:WGC-variance},  setting 
\begin{equation*}
	\beta_t = \frac{4L}{t+2}, \quad \gamma_t = \frac{3}{t+2}, \quad \eta_t = \frac{LD^2}{t(t+1)},\quad b_t = \big\lceil 3\rho t(t+1)\big \rceil
\end{equation*}
we have
\begin{equation*}\label{eq:scgs}
	\E[f(x_t) - f(x^*)] \leq \frac{6LD^2}{(t+2)^2} + \frac{15LD^2+3\| \nabla f(x^*) \| D}{(t+1)(t+2)}.
\end{equation*}
Hence, the total number of calls to the stochastic first-order oracle and linear minimization oracle required to be
solved to find an $\epsilon$-optimal point of problem \eqref{eq:sopt} are, respectively, bounded by
\begin{equation*}
	\mathcal{O}\left(\epsilon^{-1.5}\right), \quad \mathcal{O}\left(\epsilon^{-1}\right).
\end{equation*}
\item [(b)]
Assuming access to stochastic zeroth-order oracle, in addition, with Assumption~\ref{ass:WGC-moment}, \ref{ass:smooth-F}, setting 
$$ \beta_t = \frac{4L}{t+2}, \quad \gamma_t = \frac{3}{t+2}, \quad \eta_t = \frac{LD^2}{t(t+1)}, \quad b_t = \big\lceil 6\rho(d+4) t(t+1)\big \rceil, \quad \nu = \frac{D}{(T+2)^2(d+6)^{3/2}},$$
we have
\begin{equation*}\label{eq:zscgs}
	\E[f(x_t) - f(x^*)] \leq \frac{8LD^2}{(t+2)^2} + \frac{32LD^2}{(t+1)(t+2)}.
\end{equation*}
Hence, the total number of calls to the stochastic zeroth-order oracle and linear minimization oracle required to be
solved to find an $\epsilon$-optimal point of problem \eqref{eq:sopt} are, respectively, bounded by
\begin{equation*}
	\mathcal{O}\left(d\epsilon^{-1.5}\right), \quad \mathcal{O}\left(\epsilon^{-1}\right).
\end{equation*}
\end{itemize}
\end{theorem}
To the best of our knowledge, the above complexity of $\mathcal{O}(\epsilon^{-1.5})$ is not achieved for any variance reduced versions of stochastic Frank-Wolfe methods. This improvement is solely obtained by the SCGS algorithm of~\cite{lan2016conditional} under the interpolation-like assumptions which are natural in modern machine learning problems, without any variance reduction methods. We also highlight that, in the unconstrained setting,  the stochastic gradient method performs as well as its deterministic counterpart. However, the above result still falls short of the corresponding deterministic complexity of conditional gradient sliding, which is of the order $\mathcal{O}(\epsilon^{-0.5})$~\cite{lan2016conditional}. This highlights the intrinsic difficulty associated with projection-free methods for constrained stochastic optimization problems.

\section{Experiments}

\begin{figure}
    \begin{minipage}[t]{.65\linewidth}
        \centering
        \includegraphics[width=\textwidth]{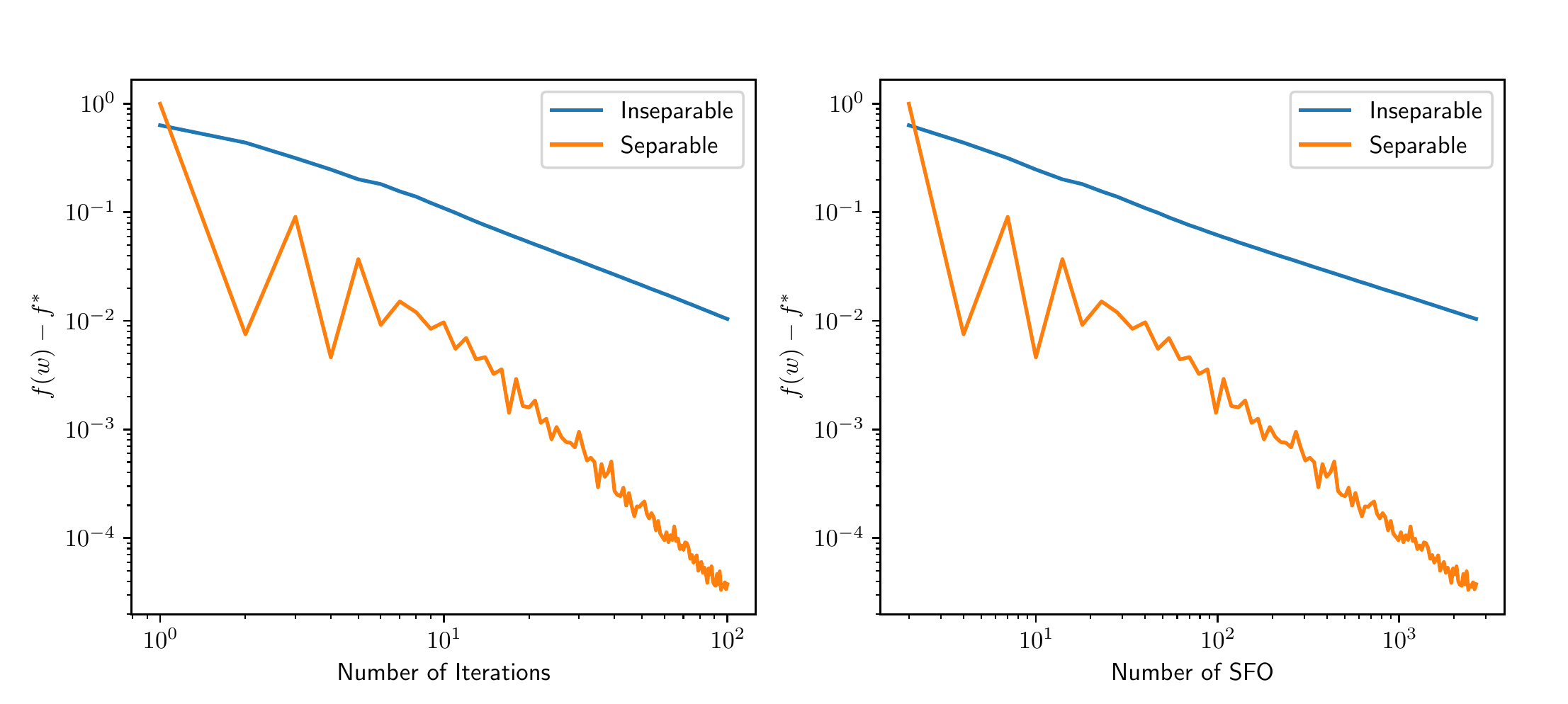}
    \end{minipage}
    \begin{minipage}[t]{.28\linewidth}
    	    \centering
        \includegraphics[width=\textwidth]{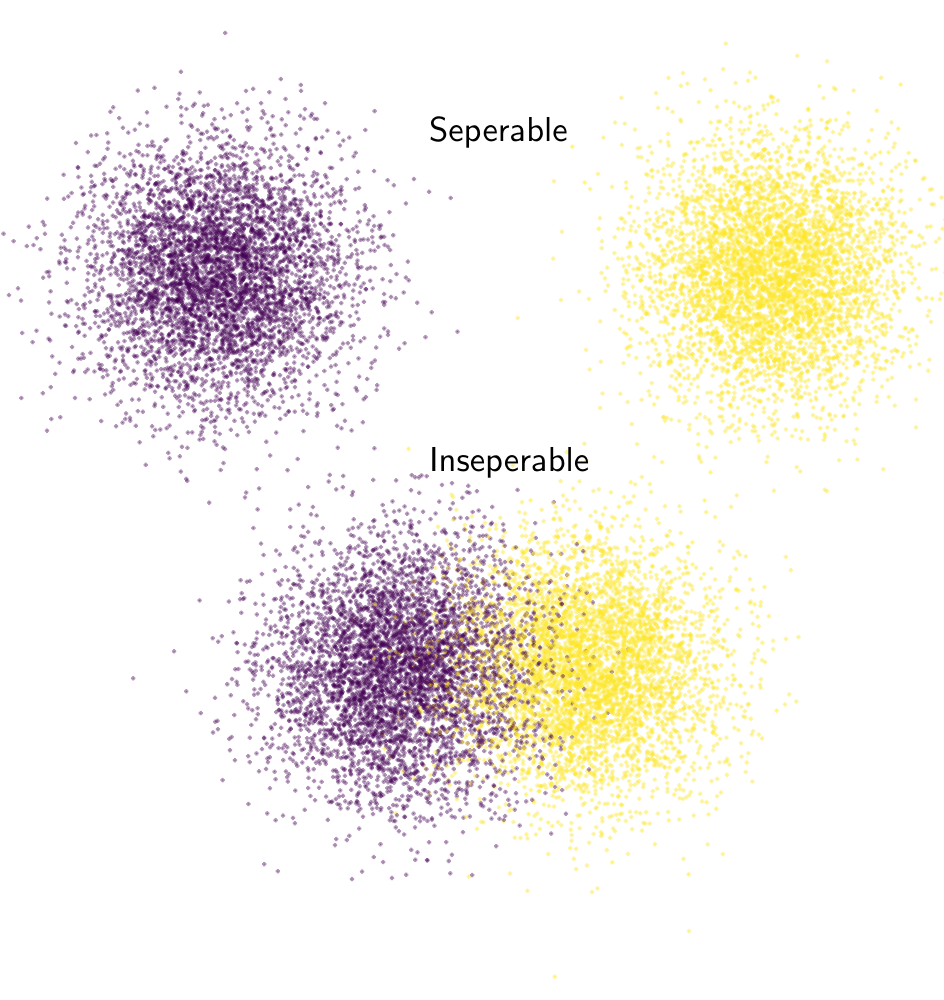}
    \end{minipage}
    \caption{The convergence behaviors of SFW for linearly (in)-separable data.  The right panel visualizes the first 2 dimensions of the synthetic data used for numerical analyses.}
     \label{fig:sfw}
\end{figure}

We generate synthetic binary classification datasets with two isotropic Gaussian blobs symmetric with respect to the origin, with the sample size $n = 100,000$ and the dimension $d=500$.  We ensure that two blobs are linearly separable with a positive margin for one dataset while the other has an overlap. We seek to find a hyperplane $w^\top x$ that minimizes the squared-hinge loss $f(w) = \frac{1}{n} \sum_{i=1}^{n} f_i(w)= \frac{1}{n} \sum_{i=1}^{n} \max (0,  1-y_i \cdot w^\top x_i)^2$ satisfying the constraint $\|w\|_1 \leq 1$.  Note that $f(w)$ satisfies the weak growth condition for linearly separable data in view of sampling only a mini-batch of gradient (with replacement) in each iteration, and the parameter $\rho = L_{\max}/L$; see Proposition 2 in \cite{vaswani2019fast}, and $L_{\max}$ is the largest Lipschitz constant for $\nabla f_i(w)$. In Figure \ref{fig:sfw},  we plot the suboptimality $f(w)-f^*$ versus the number of iterations and the number of calls to the SFO. The results are obtained by averaging over 100 runs with random initialization $w_0$. We observe that SFW converges essentially faster for linearly separable data than the inseparable case.

\section{Discussions}
%Considering convex constrained stochastic optimization problems, we show improved complexity bounds for \blue{the} vanilla stochastic conditional gradient method under certain interpolation-like conditions that occur naturally in over-parametrized models that are common in machine learning. Our results do not require any double-loop based variance reduction techniques and is hence easily implementable. Furthermore,  \blue{apart from} the batch-size parameter for stochastic conditional gradient sliding method (Algorithm~\ref{alg:SFWSlideFandZ}), the tuning parameters of the algorithm are independent of the parameter $\rho$ characterizing the interpolation-like conditions. 

We briefly discuss extensions of our results to the nonconvex setting. Our proposed assumption is motivated by the notion of Frank-Wolfe gap~\cite{DemRub70, hearn1982gap}, which is defined as $\mathcal{G}_f(x) = \max_{y\in \Omega}\langle \nabla f(x), x-y\rangle$. With this, a nonconvex function $f$ satisfies Constrained Growth Condition with constant $\rho$, if for any point $x\in \Omega$, $\E_\xi \norm{\grad F(x,\xi) - \grad f(x)}^2 \leq 2\rho L \mathcal{G}_f(x)$. Note that if $f$ is convex, then $\mathcal{G}_f (x)\geq f(x) - f(x^*)$. Hence, this generalizes Assumption~\ref{ass:WGC-variance} defined for the convex setting. Under this assumption in the nonconvex setting, it could be shown that the vanilla stochastic Frank-Wolfe algorithm can find an $\epsilon$-stationary point of the problem within at most $\mathcal{O}\left(1/\epsilon^3\right)$ and $\mathcal{O}\left(1/\epsilon^2\right)$ number of calls to the SFO linear subproblem solver, respectively. However, although existence of functions satisfying the above asssumption could be shown, it is not clear if practical nonconvex functions appearing in machine learning context satisfy it. It would be extremely interesting to examine this as future work.

%\blue{While we mainly focused on the convex setting above, in Appendix \ref{sec:nonconvexproof},  we also point out that similar improvements are possible in the nonconvex cases with a stronger growth condition assumption with respect to the Frank-Wolfe gap; see. }

\section*{Acknowledgements}
TX and KB were partially supported by a seed grant from the Center for Data Science and Artificial Intelligence Research,  UC Davis and NSF TRIPODS Grant-1934568.

\bibliography{bibtex}
\bibliographystyle{abbrv}

% In the unusual situation where you want a paper to appear in the
% references without citing it in the main text, use \nocite
%\nocite{langley00}
\clearpage

 { \hspace{1.72in} \bf {\normalsize SUPPLEMENTARY DOCUMENT}}

\appendix
\section{Proof for Theorem~\ref{thm:sfw-convex-wgc}}
In order to prove Theorem~\ref{thm:sfw-convex-wgc}, we require the following result from~\cite{nesterov2017random} for the zeroth-order case.
\begin{lemma}\textsc{\cite{nesterov2017random}}
Let the function $f$ has lipschitz continuous gradient with constant $L$. Consider the smoothed function $f_\nu (x) = \E_{u} [f(x+\nu u)]$ where $u\sim \normal(0, \mathbf{I}_d)$. Then for any $x\in\reals^d$, 
	\begin{align}
		\E_u \left[\frac{f(x+\nu u) - f(x)}{\nu}u\right] & = \grad f_\nu (x)\label{eq:nestrov-first-moment}\\
		\norm{ \grad f_{\nu}(x) - \grad f(x) } &\leq \frac{\nu}{2}L(d+3)^{\frac{3}{2}}\label{eq:nestrov-bias}\\
		\frac{1}{\nu^2} \E_u [\{f(x+\nu u) - f(x)\}^2 \norm{u}^2] &\leq \frac{\nu^2}{2}L^2 (d+6)^3 + 2(d+4) \norm{\grad f(x)}^2. \label{eq:nestrov-second-moment}
	\end{align}
\end{lemma}

We now present the lemma below to bound the mean squared error for the zeroth-order gradient estimator. 

\begin{lemma}\label{lem:zero_grad}
Under Assumption \ref{ass:constrained}, \ref{ass:smooth-F}, \ref{ass:WGC-moment}, we have
\begin{align}
	\label{eq:zero_grad_var}	\E \norm{\bar{G}_{\nu}^{t} - \grad f_{\nu}(x_{t-1}) }^2 & \leq \frac{ 4\rho L(d+4)(f(x_{t-1}) - f(x^*)) }{b_t} + \frac{\nu^2 L^2 (d+6)^3} {2b_t},\\
	\label{eq:zero_grad_mse}	\E \norm{\bar{G}_{\nu}^{t} - \grad f(x_{t-1}) }^2 & \leq \frac{ 4\rho L(d+4)(f(x_{t-1}) - f(x^*)) }{b_t} + \nu^2 L^2 (d+6)^3.
\end{align}
\end{lemma}

\begin{proof}
First note that by \eqref{eq:nestrov-first-moment}, we have 
\begin{equation*}
	\E_{u,\xi} [\bar{G}_{\nu}^{t}] = \E_{u,\xi} [G_{t,j}] = \E_{u} \left[\frac{f(x_{t-1}+\nu u) - f(x_{t-1})}{\nu}u\right] = \grad f_\nu (x_{t-1}),
\end{equation*}

Then by using \eqref{eq:nestrov-second-moment} for $F$ instead of $f$, under Assumption \ref{ass:smooth-F}, \ref{ass:WGC-moment}, we can obtain
\begin{flalign*}
	\E_{u, \xi} \norm{\bar{G}_{\nu}^{t} - \grad f_{\nu}(x_{t-1}) }^2 & = \frac{1}{b_t} \E_{u,\xi} \norm{G_{t,j} - \grad f_{\nu}(x_{t-1})}^2\\
	&\leq \frac{1}{b_t} \E_{u, \xi} \norm {G_{t,j}}^2\\
	&\leq \frac{2(d+4)}{b_t}\E_{\xi} \norm{\grad F(x_{t-1}, \xi_{t,j})}^2 + \frac{\nu^2 L^2 (d+6)^3}{2b_t}\\
	&\leq  \frac{ 4\rho L(d+4)(f(x_{t-1}) - f(x^*)) }{b_t} + \frac{\nu^2 L^2 (d+6)^3} {2b_t}
\end{flalign*}
where the first inequality comes from the fact that the variance is less than the seocond moment. 

To prove \eqref{eq:zero_grad_mse}, we decompose the mean squared error into the bias and the variance by utilizing the results \eqref{eq:zero_grad_var} and \eqref{eq:nestrov-bias}, i.e.,
\begin{flalign*}
	&& \E \norm{\bar{G}_{\nu}^{t} - \grad f(x_{t-1}) }^2  & = \E \norm{\bar{G}_{\nu}^{t} - \grad f_{\nu}(x_{t-1}) }^2 + \norm{\grad f_{\nu}(x_{t-1}) - \grad f(x_{t-1})}^2\\
	&& &\leq \frac{ 4\rho L(d+4)(f(x_{t-1}) - f(x^*)) }{b_t} + \frac{\nu^2 L^2 (d+6)^3} {2b_t} + \frac{\nu^2L^2(d+3)^{3}}{4}\\
	&& &\leq  \frac{ 4\rho L(d+4)(f(x_{t-1}) - f(x^*)) }{b_t} + \nu^2 L^2 (d+6)^3.% &\rhd \text{ by $b_t\geq 1$}
\end{flalign*}
\end{proof}

We also need the following simple result in our proof.
\begin{lemma}\label{gamma_recursive}
Assume that sequences $\{\phi_t\}_{t \ge 0} \ge 0$, $\{B_t\}_{t \ge 1}$, $\{\theta_t\}_{t \ge 1} \in [0,1]$ are given such that
\begin{equation}\label{rec1}
\phi_t \le (1-\theta_t) \phi_{t-1} + B_t.
\end{equation}
Then, we have
\[
\phi_T \le \Theta_T \left[\phi_0 + \sum_{t=1}^ T \frac{B_t}{\Theta_t}\right],
\]
where, for any $t \ge 2$, 
\begin{equation}\label{def_gamma}
\Theta_t = \Theta_1 \prod_{k=2}^{t} (1-\theta_k), \quad \text{where} \ \ \Theta_1 = 1-\theta_1 \ \  \text{if} \ \ \theta_1<1, \quad \Theta_1 = 1 \ \  \text{if} \ \ \theta_1=1.
\end{equation}

\end{lemma}

\begin{proof}
Dividing both sides of \eqref{rec1} by $\Theta_t$, summing them up from $t=1$ to $t=T$, noting non-negativity of $\phi_t$ and \eqref{def_gamma}, we obtain the result.
\end{proof}

\begin{proof}[Proof for Theorem \ref{thm:sfw-convex-wgc}] For convenience, let $g_t$ be the gradient estimator at $t$ step. Thus, $g_t = \tilde{\grad}_t$ for the first order method while in the zeroth order setting $g_t = \bar{G}_{\nu}^t$.
\begin{flalign*}
    f(x_t) &\leq f(x_{t-1}) +  \inner{\grad f(x_{t-1}), x_t-x_{t-1}} + \frac{L}{2} \norm{x_t - x_{t-1}}^2\\ %& \rhd \text{ by the smoothness}\\
    &= f(x_{t-1}) +  \gamma_t\inner{\grad f(x_{t-1}), d_t-x_{t-1}} + \frac{L\gamma_t^2}{2} \norm{d_t- x_{t-1}}^2 \\
    &\leq f(x_{t-1}) + \gamma_t \inner{g_t, d_t-x_{t-1}} + \gamma_t\inner{\grad f(x_{t-1}) - g_t, d_t-x_{t-1}} + \frac{L D^2\gamma_t^2 }{2} \\
    &\leq f(x_{t-1}) + \gamma_t \inner{g_t, x^*-x_{t-1}} + \gamma_t\inner{\grad f(x_{t-1}) - g_t, d_t-x_{t-1}} + \frac{L D^2\gamma_t^2 }{2} \\%& \rhd \text{ by the optimality}\\
    &=  f(x_{t-1}) + \gamma_t \inner{\grad f(x_{t-1}), x^*-x_{t-1}} + \gamma_t\inner{\grad f(x_{t-1}) - g_t, d_t-x^*} + \frac{L D^2\gamma_t^2 }{2}\\
    &\leq f(x_{t-1}) + \gamma_t (f(x^*)- f(x_{t-1})) + \gamma_t\inner{\grad f(x_{t-1}) - g_t, d_t-x^*} + \frac{L D^2\gamma_t^2 }{2} \\%& \rhd \text{ by the convexity}\\
    &\leq f(x_{t-1}) + \gamma_t (f(x^*)- f(x_{t-1})) + \frac{\gamma_t}{2\beta}\norm{\grad f(x_{t-1}) - g_t}^2 + \frac{ D^2\gamma_t(L\gamma_t + \beta) }{2}.% & \rhd \text{ for any $\beta>0$}
\end{flalign*}
The last inequality comes from the Young's inequality: for any $\beta >0$,
\begin{align*}
    \inner{\grad f(x_{t-1}) - g_t, d_t-x^*} &\leq \frac{1}{2\beta} \norm{\grad f(x_{t-1}) - g_t}^2 + \frac{\beta}{2} \norm{d_t-x^*}^2 \\
    &\leq \frac{1}{2\beta} \norm{\grad f(x_{t-1}) - g_t}^2 + \frac{D^2\beta}{2}.
\end{align*}

Denote $\phi_t = f(x_t) - f(x^*)$. Substracting $f(x^*)$ from both sides of the inequality and taking the conditional expectation $\E[ \cdot | \mathcal{F}_{t-1}]$, we have
\begin{equation}\label{eq:fw-proof}
    \E[\phi_t | \mathcal{F}_{t-1}]  \leq (1-\gamma_t)\phi_{t-1} + \frac{\gamma_t}{2\beta} \E [\norm{\grad f(x_{t-1}) - g_t}^2 | \mathcal{F}_{t-1}] + \frac{D^2 \gamma_t (L\gamma_t + \beta)}{2}.
\end{equation}
 For the first-order gradient estimator $g_t = \tilde{\grad}_t$, we have the following bound for its variance under Assumption \ref{ass:WGC-variance}:
\begin{equation*}
	\E [\norm{\grad f(x_{t-1}) - \tilde{\grad}_t}^2 | \mathcal{F}_{t-1}] = \frac{1}{b_t} \E[\norm{\grad f(x_{t-1}) - \grad F(x_{t-1}, \xi_{t,j})}^2| \mathcal{F}_{t-1}]\leq \frac{2\rho L\phi_{t-1}}{b_t}.
\end{equation*}
Then by \eqref{eq:fw-proof}, we can obtain
\begin{equation*}
    \E[\phi_t|\mathcal{F}_{t-1}]\leq (1-\gamma_t)\phi_{t-1} + \frac{\gamma_t\rho L}{\beta b_t} \phi_{t-1} + \frac{D^2 \gamma_t (L\gamma_t + \beta)}{2}.
\end{equation*}
Let $\gamma_t = \frac{4}{t+3}, \beta = \rho L \gamma_t = \frac{4\rho L}{t+3}> 0, b_t = \lceil(t+3)/2\rceil$, then
\begin{equation}\label{eq:induction}
     \E[\phi_t | \mathcal{F}_{t-1}] \leq\left(1-\frac{2}{t+3}\right) \phi_{t-1} + \frac{8(\rho+1)LD^2}{(t+3)^2}.
\end{equation}

Now, letting $\theta_t = \frac{2}{t+3}$, it is easy to check that $\Theta_t = \frac{6}{(t+2)(t+3)}$ due to \eqref{def_gamma}. Hence, in the view of Lemma~\ref{gamma_recursive}, we have
\[
\E[\phi_t ] \leq \frac{6\phi_0}{(t+2)(t+3)} + \frac{8(\rho+1)LD^2}{t+3} \le \frac{2[\phi_0+4(\rho+1)LD^2]}{t+3}.
\]
The above inequality implies that to attain an $\epsilon$-optimal point, the total number of interations $T$ can be bounded by $\mathcal{O}(1/\epsilon).$ Hence, the number of the gradient calls $\sum_{t=1}^T b_t$ can be bounded by $\frac{T^2+7T}{4} = \mathcal{O}(T^2)$, and the number of calls to the linear minimization oracle immediately follows from this observation.

\noindent We now prove part (b). For the zeroth-order version, by \eqref{eq:zero_grad_mse} in Lemma \ref{lem:zero_grad} and \eqref{eq:fw-proof}, we can obtain
\begin{align*}
\E[\phi_t | \mathcal{F}_{t-1}] &\leq (1-\gamma_t)\phi_{t-1} + \frac{\gamma_t}{2\beta} \E [\norm{\grad f(x_{t-1}) - \bar{G}_\nu^t}^2 | \mathcal{F}_{t-1}] + \frac{D^2 \gamma_t (L\gamma_t + \beta)}{2}\\
    &\leq (1-\gamma_t)\phi_{t-1} + \frac{2\gamma_t\rho L(d+4)}{\beta b_t} \phi_{t-1} +  \frac{\gamma_t\nu^2 L^2 (d+6)^3}{2\beta} +\frac{D^2 \gamma_t (L\gamma_t + \beta)}{2}
\end{align*}
Let $\gamma_t = \frac{4}{t+3}, \beta = \gamma_t \rho L, b_t = (t+3)(d+4), \nu =  D(T+3)^{-1}(d+6)^{-3/2}\leq  D(t+3)^{-1}(d+6)^{-3/2}$, then we have
\begin{equation*}
\E[\phi_t | \mathcal{F}_{t-1}] \leq \left(1-\frac{2}{t+3}\right)\phi_{t-1}  + \frac{8(\rho+\rho^{-1}+1)}{(t+3)^2 LD^2}
\end{equation*}
Similarly, in the view of Lemma~\ref{gamma_recursive}, we obtain
\[
\E[f(x_t) - f(x^*)] \leq \frac{2[f(x_0) - f(x^*)]+8(\rho+\rho^{-1}+1)LD^2}{t+3}.
\]
The above inequality implies that to attain an $\epsilon$-optimal point, the total number of interations $T$ can be bounded by $\mathcal{O}(1/\epsilon).$ Hence, the number of calls to the zeroth-order oracles $2\sum_{t=1}^T b_t$ can be bounded by $(d+4)(T^2+7T) = \mathcal{O}(dT^2)$, and the number of calls to the linear minimization oracle immediately follows from this observation.
\end{proof}

\section{Proof of Theorem~\ref{thm:scgs-convex-wgc}}

\begin{proof} 
[Proof of Theorem \ref{thm:scgs-convex-wgc}] 
For convenience, let $g_t$ be the gradient estimator at $t$ step. Thus, $g_t = \tilde{\grad}_t$ for the first order method while in the zeroth order setting $g_t = \bar{G}_{\nu}^t$. First note that by the updates in Algorithm~\ref{alg:SFWSlideFandZ}, the convexity and the smoothness of $f$, we have

\begin{align}
	f(x_{t})  \leq & f(z_t) + \inner{\nabla f(z_t), x_t-z_t} +\frac{L}{2}\norm{x_t-z_t}^2\nonumber\\
	= &(1-\gamma_t) [f(z_{t})  +\inner{\nabla f(z_t), x_{t-1} - z_t }]+\gamma_t [f(z_{t}) + \inner{\nabla f(z_t), y_t-z_{t}}] +\frac{L\gamma_t^2}{2}\norm{y_t-y_{t-1}}^2\nonumber \\
	\leq & (1-\gamma_t) f(x_{t-1}) +\gamma_t [f(z_{t}) + \inner{\nabla f(z_t), y_t-z_{t}}] +\frac{L\gamma_t^2}{2}\norm{y_t-y_{t-1}}^2 \nonumber \\
	= &(1-\gamma_t) f(x_{t-1}) +\gamma_t [f(z_{t}) + \inner{\nabla f(z_t), y_t-z_{t}}] +\frac{\beta_t\gamma_t}{2}\norm{y_t-y_{t-1}}^2\nonumber \\
	& \quad - \frac{\gamma_t(\beta_t - L\gamma_t)}{2}\norm{y_t-y_{t-1}}^2. \label{eq:scgs-proof-1}
\end{align}

And by~\eqref{eq:acc_subprob}, we have
\begin{equation*}
\inner{g_t + \beta_t(y_t-y_{t-1}), y_t -x }\leq \eta_t, \quad \forall x\in\Omega.
\end{equation*}
Let $x=x^*$ in the above inequality. Then we have
\begin{align}
\frac{1}{2}\norm{y_t- y_{t-1}}^2 & = \frac{1}{2}\norm{y_{t-1}-x^*}^2 - \inner{y_{t-1}-y_{t}, y_t-x^*} - \frac{1}{2}\norm{y_t-x^*}^2 \nonumber\\
&\leq \frac{1}{2}\norm{y_{t-1}-x^*}^2 + \frac{1}{\beta_t}\inner{g_t, x^*-y_t} - \frac{1}{2}\norm{y_t-x^*}^2 + \frac{\eta_t}{\beta_t}. \label{eq:scgs-proof-2}
\end{align}

Denoting $\delta_t = g_t- \nabla f(z_t)$ and combining \eqref{eq:scgs-proof-1} and \eqref{eq:scgs-proof-2}, we obtain
\begin{align*}
	f(x_t)	 \leq & (1-\gamma_t) f(x_{t-1}) + \gamma_t f(x^*) + \gamma_t \inner{\delta_t, x^*-y_t}  \\
	&+ \frac{\beta_t\gamma_t}{2}(\norm{y_{t-1}-x^*}^2 -\norm{y_{t}-x^*}^2) +  \eta_t\gamma_t -\frac{\gamma_t}{2} (\beta_t - L\gamma_t) \norm{y_t - y_{t-1}}^2\\
	 = & (1-\gamma_t) f(x_{t-1}) + \gamma_t f(x^*)  + \frac{\beta_t\gamma_t}{2}(\norm{y_{t-1}-x^*}^2 -\norm{y_{t}-x^*}^2)+\eta_t\gamma_t\\
	 &+\gamma_t \inner{\delta_t, x^*-y_{t-1}} + \gamma_t \inner{\delta_t, y_{t-1}-y_t}  - \frac{\gamma_t}{2}(\beta_t - L\gamma_t)\norm{y_{t}-y_{t-1}}^2\\
	 \leq &(1-\gamma_t) f(x_{t-1}) + \gamma_t f(x^*)  + \frac{\beta_t\gamma_t}{2}(\norm{y_{t-1}-x^*}^2 -\norm{y_{t}-x^*}^2)+\eta_t\gamma_t\\
	 &+\gamma_t \inner{\delta_t, x^*-y_{t-1}} + \frac{\gamma_t\norm{\delta_t}^2}{2(\beta_t-L\gamma_{t})},
\end{align*}
where the last inequality comes from the fact that
\begin{equation*}
	\gamma_t \inner{ \delta_t, y_{t-1} - y_t }  \leq   \frac{\gamma_t}{2(\beta_t - L\gamma_t)} \norm{ \delta_t}^2 + \frac{\gamma_t(\beta_t - L\gamma_t)}{2}  \norm{y_t - y_{t-1}}^2.
\end{equation*}

Substracting $f(x^*)$ from both sides of the above inequality, denoting $\phi_t = f(x_t) - f(x^*)$, $\theta_t=\gamma_t$, and in the view of Lemma~\ref{gamma_recursive}, we obtain
\begin{equation}\label{eq:scgs-proof-basic-ineq}
	\phi_t \leq \Theta_t\left[\phi_0 +\sum_{k=1}^{t} \frac{B_k}{\Theta_k}\right],
\end{equation}
where
\begin{align*}
B_t &= \frac{\beta_t\gamma_t}{2}(\norm{y_{t-1}-x^*}^2 -\norm{y_{t}-x^*}^2)+\eta_t\gamma_t +\gamma_t \inner{\delta_t, x^*-y_{t-1}} + \frac{\gamma_t\norm{\delta_t}^2}{2(\beta_t-L\gamma_{t})}.
\end{align*}
Choosing $\gamma_t=\theta_t = \frac{3}{t+2}$, we can easily check that $\Theta_t = \frac{6}{t(t+1)(t+2)}$ due to \eqref{def_gamma}. 
Moreover, letting $\beta_t = \frac{4L}{t+2}, \eta_t = \frac{LD^2}{t(t+1)}$, we have $\sum_{k=1}^{t} \frac{\eta_k\gamma_k}{\Theta_k}\leq \frac{tLD^2}{2}$ and

\begin{align*}
	& \sum_{k=1}^{t} \frac{\beta_k\gamma_k}{\Theta_k} (\norm{y_{t-1}-x^*}^2 -\norm{y_{t}-x^*}^2) \nonumber \\
	& \leq \frac{\beta_1\gamma_1}{\Theta_i} \norm{y_0 - x^*}^2 + \sum_{k=2}^{t} \bigg( \frac{\beta_k\gamma_k}{\Theta_k} - \frac{\beta_{k-1}\gamma_{k-1}}{\Theta_{k-1}} \bigg) \norm{y_{t-1}-x^*}^2\\
	&\leq \frac{\beta_1\gamma_1}{\Theta_i} D^2 + \sum_{k=2}^{t} \bigg( \frac{\beta_k\gamma_k}{\Theta_k} - \frac{\beta_{k-1}\gamma_{k-1}}{\Theta_{k-1}} \bigg) D^2 = \frac{\beta_t\gamma_t D^2}{\Theta_t} = \frac{2LD^2t(t+1)}{t+2},
\end{align*}
where the last inequality comes from the fact $\frac{\beta_k\gamma_k}{\Theta_k} > \frac{\beta_{k-1}\gamma_{k-1}}{\Theta_{k-1}}$.\\

We now prove part (a). Let $g_t = \tilde{\nabla}_t$. Taking expectation for both sides of \eqref{eq:scgs-proof-basic-ineq}, and noting that $	\E [\inner{\delta_t, x^*-y_{t-1}}]  = 0$ and
\begin{flalign*}
&&\E[\norm{\delta_t}^2|\mathcal{F}_{t-1}] & \leq \frac{2\rho L}{b_t} (f(z_t) - f(x^*)) \quad  \rhd \text{ by Assumption \ref{ass:WGC-variance}}\\
&& & \leq \frac{2\rho L}{b_t} \bigg( (1-\gamma_t) \phi_{t-1} + \gamma_t (f(y_{t-1})  - f(x^*) )  \bigg) \quad \rhd \text{$z_t = (1-\gamma_t)x_{t-1}  + \gamma_t y_{t-1} $}\\
&& & \leq \frac{2\rho L}{b_t} \bigg( (1-\gamma_t) \phi_{t-1} + \gamma_t(\|\nabla f(x^*) \|D + \frac{LD^2}{2})  \bigg) \quad \rhd \text{ by the smoothness}\\
&& & := \frac{2\rho L}{b_t} \bigg( (1-\gamma_t) \phi_{t-1} + \gamma_t \frac{K LD^2}{2}  \bigg),\quad \rhd  K = \frac{\| \nabla f(x^*) \|}{LD} +1
\end{flalign*}
we can obtain
\begin{equation*}
	\E [\phi_t] \leq \frac{6LD^2}{(t+2)^2} + \frac{3LD^2}{(t+1)(t+2)} + \frac{3}{t(t+1)(t+2)} \sum_{k=1}^{t}\frac{\rho k(k+1)\bigg( (k-1)\E [\phi_{k-1}]  + \frac{3KLD^2}{2}\bigg)}{b_k}
\end{equation*}

We now prove
\begin{equation}\label{eq:scgs-ind}
    	\E [\phi_t]  \leq \frac{6LD^2}{(t+2)^2} + \frac{(12+3K)LD^2}{(t+1)(t+2)}
\end{equation}
by induction. Set $b_k = \big\lceil 3\rho k(k+1)\big \rceil$. It is easy to check $\E[\phi_0] \leq \frac{KLD^2}{2} $ by the smoothness of $f$ which satisfies \eqref{eq:scgs-ind}. If \eqref{eq:scgs-ind} holds for all $k\leq t-1$, then with the above inequality we can obtain
\begin{align*}
	\E [\phi_t] &\leq \frac{6LD^2}{(t+2)^2} + \frac{(3+\frac{3K}{2})LD^2}{(t+1)(t+2)} + \frac{1}{t(t+1)(t+2)} \sum_{k=1}^{t}(k-1)\E [\phi_{k-1}]\\
	&\leq \frac{6LD^2}{(t+2)^2} + \frac{(3+\frac{3K}{2})LD^2}{(t+1)(t+2)} + \frac{1}{t(t+1)(t+2)} \sum_{k=1}^{t}\bigg( \frac{6LD^2(k-1)}{(k+1)^2} + \frac{(12+3K)LD^2(k-1)}{k(k+1)} \bigg)\\
	&\leq \frac{6LD^2}{(t+2)^2} + \frac{(3+\frac{3K}{2})LD^2}{(t+1)(t+2)} + \frac{(18+3K) LD^2}{t(t+1)(t+2)} \sum_{k=1}^{t} \frac{1}{k+1} \\
	&\leq \frac{6LD^2}{(t+2)^2} + \frac{(12+3K)LD^2}{(t+1)(t+2)},
\end{align*}
i.e., \eqref{eq:scgs-ind} holds for $k=t$. Therefore, to achieve an $\epsilon$-optimal point, the number of outer iterations $T$ can be bounded by $\mathcal{O}(1/\sqrt{\epsilon})$. Hence, the number of calls to the first order oracles can be bounded by
\begin{equation*}
	\sum_{t=1}^T b_t \leq 3\rho\sum_{t=1}^T t(t+1) = \rho T(T+1)(T+2) =\mathcal{O}(T^3).
\end{equation*}
Due to the fact that the inner iterations indeed solves a convex constrained optimization problem by the classical Frank-Wolfe method with the exact line search, one can show that the number of inner iterations $N_t$ performed at the $t$-th out iteration can be bounded by
\begin{equation*}
	N_t \leq \left\lceil \frac{6\beta_tD^2}{\eta_t} \right\rceil = \mathcal{O}(t).
\end{equation*}
Thus, the number of calls to the linear minimization oracle can be bounded by
\begin{equation*}
	\sum_{t=1}^T N_t \leq \mathcal{O}(T^2).
\end{equation*}

We now prove part (b). Let $g_t = \bar{G}_\nu^t$. Notice that $\bar{G}_\nu^t$ is a biased estimator of $\nabla f(z_t)$. We can obtain the following results by \eqref{eq:nestrov-bias}:
\begin{align*}
	\E[ \inner{\delta_t, x^* - y_{t-1}} ] &= \E[ \inner{\nabla f_{\nu}(z_t) - \nabla f(z_t), x^*- y_{t-1}}] + \E[ \inner{\bar{G}_\nu^t - \nabla f_{\nu}(z_t), x^*- y_{t-1}}]\\
	& =  \E[ \inner{\nabla f_{\nu}(z_t) - \nabla f(z_t), x^*- y_{t-1}}]\leq \frac{\nu LD(d+3)^{3/2}}{2}.
\end{align*}
Besides, we can obtain a similar bound for $\E[\norm{\delta_t}^2]$ by Lemma \ref{lem:zero_grad}.
\begin{align*}
	\E[\norm{\delta_t}^2|\mathcal{F}_{t-1}] &\leq  \frac{ 4\rho L(d+4)(f(z_{t}) - f(x^*)) }{b_t} + \nu^2 L^2 (d+6)^3\\
	&\leq  \frac{ 4\rho L(d+4)\big( (1-\gamma_t) \phi_{t-1} + \frac{LD^2\gamma_t}{2}  \big) }{b_t} + \nu^2 L^2 (d+6)^3.
\end{align*}
where the last inequality is slightly different from the one for the first-order setting due to  $\| \nabla f(x^*) \| = 0$ for convex cases under the moment-based WGC.

By \eqref{eq:scgs-proof-basic-ineq}, we have the following simplified inequality:
\begin{align*}
	\E[\phi_t] \leq & \frac{6LD^2}{(t+2)^2} + \frac{3LD^2}{(t+1)(t+2)} + \frac{\nu LD(d+3)^{3/2}}{2} + \frac{3\nu^2L(d+6)^3}{2} \\&~~+ \frac{6}{t(t+1)(t+2)} \sum_{k=1}^{t}\frac{\rho(d+4) k(k+1)\bigg( (k-1)\E [\phi_{k-1}]  + \frac{3LD^2}{2}\bigg)}{b_k}.
\end{align*}
Set $b_k = \lceil 6\rho k(k+1)(d+4) \rceil, \nu =\frac{D}{(T+2)^2(d+6)^{3/2}}\leq \frac{D}{(t+2)^2(d+6)^{3/2}}$. Then we have
\begin{equation*}
	\E[\phi_t] \leq \frac{8LD^2}{(t+2)^2} + \frac{12LD^2}{(t+1)(t+2)} + \frac{1}{t(t+1)(t+2)} \sum_{k=1}^{t} (k-1)\E [\phi_{k-1}].
\end{equation*}
Similar to the proof for part (a), we can finish the proof by induction and obtain the bounds for complexity.
\end{proof}

\section{Zeroth-order SGD under Growth Conditions}\label{sec:sgdproofs}

In this section, we highlight that one can extend the results in \cite{vaswani2019fast} only assuming access to stochastic zeroth-order oracle with corresponding variance-based growth conditions. Notice that both SGC and WGC are defined in the format of the relative shrinkage of $\E \|\nabla F(x,\xi) \|^2$. However, in the unconstrained setting, the corresponding variance-based versions are equivalent to the moment-based growth conditions (see Proposition~\ref{prop:WGC} for WGC; for SGC, note that $\E \| \nabla F(x, \xi) - \nabla f(x)  \|^2 = \E \| \nabla F(x, \xi) \|^2 - \|  \nabla f(x) \|^2 =(\rho-1)\| \nabla f(x) \|^2$).

We present the following result for the zeroth-order setting which directly follows the proofs in \cite{vaswani2019fast}. We highlight that it is the zeroth-order version of Theorem 3 in \cite{vaswani2019fast}. Similar results for other setups considered in~\cite{vaswani2019fast} can also be obtained for the zeroth-order setting.
\begin{algorithm}[h]
\caption{Non-convex Zeroth-order SGD (ZO-SGD)}
\begin{algorithmic}
 \STATE \textbf{Input:} $x_0\in \Omega$, number of iterations $T$, $\eta$
    \FOR{$t = 1, 2, \dots, T$}
        \STATE  Randomly pick $\xi_{t}$ and compute
        $$x_{t} = x_{t-1} - \eta \frac{F(x_{t-1} + \nu u_t, \xi_t ) - F(x_{t-1}, \xi_t)}{\nu} u_{t} := x_{t-1} - \eta G_{t}.  $$
        where $u_{t}$ is generated from $\normal(0, \mathbf{I}_d)$.
    \ENDFOR
    \STATE \textbf{Output:} $x_R$ where $R$ is uniformly distributed over ${0, \dots , T-1}$
\end{algorithmic}
\label{alg:ZO-SGD}
\end{algorithm}

\begin{theorem}
Consider solving the non-convex unconstrained $L$-smooth problem by Algorithm~\ref{alg:ZO-SGD} with some appropriate constant step size $\eta$, if $f$ satistifies SGC with constant $\rho$, then
\begin{equation*}
    \E \| \grad f(x_R) \|^2  \leq \mathcal{O}\left(\frac{1}{T}\right)
\end{equation*}
\end{theorem}
\begin{proof-idea}
The zeroth-order SGD update is given by $$x_{t} = x_{t-1} - \eta \frac{F(x_{t-1} + \nu u_t, \xi_t ) - F(x_{t-1}, \xi_t)}{\nu} u_{t} := x_{t-1} - \eta G_{t}.  $$

By the smoothness of $f$, we have
\begin{align*}
	f(x_{t}) - f(x_{t-1}) &\leq \langle \nabla f(x_{t-1}, x_t - x_{t-1}) + \frac{L}{2}\| x_t -x_{t-1} \|^2\\
	&= -\eta \langle \nabla f(x_{t-1}) , G_t \rangle + \frac{L\eta^2}{2} \| G_t \|^2
\end{align*}
 Consider the term $\langle \nabla f(x_{t-1}) , G_t \rangle$. Taking expectation with respect to $\xi_t,u_t$, we have
\begin{align*}
	\E [\langle \nabla f(x_{t-1}) , G_t \rangle ]  &= \langle \nabla f(x_{t-1}), \nabla f_\nu (x_{t-1}) \rangle\\
	& = \langle \nabla f(x_{t-1}), \nabla f(x_{t-1}) + \nabla f_\nu (x_{t-1}) - \nabla f(x_{t-1}) \rangle\\
	&\geq \| \nabla f(x_{t-1}) \|^2  - \frac{\nu L(d+3)^{3/2}}{2}\| \nabla f(x_{t-1}) \|
\end{align*} 
 Consider the term $\| G_t \|^2$. Taking expectation with respect to $\xi_t,u_t$, we have
\begin{align*}
	\E \| G_t \|^2 &\leq \frac{\nu^2}{2}L^2 (d+6)^3 + 2(d+4) \E \| \nabla F(x_{t-1}, \xi_t) \|^2\\
	&\leq \frac{\nu^2}{2}L^2 (d+6)^3 + 2\rho (d+4) \| \nabla f(x_{t-1}) \|^2
\end{align*} 
Then, by the above inequalities, we can obtain
\begin{align*}
	&\E [ f(x_t) - f(x_{t-1}) ] \\
	&\leq -\eta \| \nabla f(x_{t-1}) \|^2 + \eta\frac{\nu L(d+3)^{3/2}}{2}\| \nabla f(x_{t-1}) \| +\eta^2 L\rho(d+4) \| \nabla f(x_{t-1}) \|^2 + \eta^2\frac{\nu^2}{4}L^3(d+6)^2\\
	&\leq  -\eta \| \nabla f(x_{t-1}) \|^2 + \eta^2L(d+3) \| \nabla f(x_{t-1}) \|^2 + \frac{\nu^2 L(d+3)^2}{16}  \\ &\quad\quad +\eta^2 L\rho(d+4) \| \nabla f(x_{t-1}) \|^2 + \eta^2\frac{\nu^2}{4}L^3(d+6)^2\\
	&\leq -\eta \| \nabla f(x_{t-1}) \|^2 + \eta^2L(\rho+1)(d+4) \| \nabla f(x_{t-1}) \|^2 + \eta^2\frac{\nu^2}{4}L^3(d+6)^2 + \frac{\nu^2 L(d+3)^2}{16}.
\end{align*}
If $\eta = \frac{1}{2L(\rho+1)(d+4)}$, then we have
\begin{align*}
	\E [ f(x_t) - f(x_{t-1}) ] \leq -\frac{\eta}{2}  \| \nabla f(x_{t-1}) \|^2 + \eta^2\frac{\nu^2}{4}L^3(d+6)^2 + \frac{\nu^2 L(d+3)^2}{16}\\
	\Rightarrow  \| \nabla f(x_{t-1}) \|^2 \leq \frac{2}{\eta}\E [ f(x_{t-1}) - f(x_{t}) ] + \eta\frac{\nu^2}{2}L^3(d+6)^2 + \frac{\nu^2 L(d+3)^2}{8\eta}
\end{align*}
Setting $\nu = \mathcal{O}(1/\sqrt{dT})$ and taking a telescoping sum of the above inequality, we can get the same $\mathcal{O}(1/T)$ rate for the non-convex setting.
\end{proof-idea}
 In the above proof, we did not pay careful attention to the exact constants of the tuning parameter, as our main point is to simply highlight it is possible to obtain a zeroth-order version of the results in~\cite{vaswani2019fast} under variance-based growth conditions and the logic of the proof is the same as~\cite{vaswani2019fast}.

\end{document}